\documentclass[10pt]{article}
\usepackage{amsmath}
\usepackage{amsfonts}
\usepackage{amsthm}
\usepackage{amssymb}
\usepackage{color}

\newtheorem{theorem}{Theorem}[section]

\newtheorem{lemma}[theorem]{Lemma}
\newtheorem{remark}[theorem]{Remark}

\numberwithin{equation}{section}



\title{Existence of a nontrival solution for Dirichlet problem involving $p(x)$-Laplacian}

\author{\small
  \small Sylwia Barna\'s\\
  \small email: Sylwia.Barnas@im.uj.edu.pl\\
\\
  \small Cracow University of Technology\\
  \small Institute of Mathematics\\
  \small ul. Warszawska 24,\
  \small 31-155 Krak\'ow, Poland\\
\\
  \small Jagiellonian University\\
  \small Faculty of Mathematics and Computer Science\\
  \small ul. {\L}ojasiewicza 6,\
  \small 30-348 Krak\'ow, Poland}
\date{}
\begin{document}

\maketitle
\bibliographystyle{plain}

\noindent \textbf{Abstract:}
  In this paper we study the nonlinear Dirichlet problem involving
  $p(x)$-Laplacian (hemivariational inequality) with nonsmooth potential. By using nonsmooth critical point theory for locally Lipschitz functionals due to Chang \cite{changg} and the properties of variational Sobolev spaces, we establish conditions which ensure the existence of solution for our problem.


\section{Introduction} \label{Intr}
Let $\Omega\subseteq\mathbb{R}^N$ (where $N>2$) be a bounded domain with a $\mathcal{C}^2$-boundary $\partial \Omega$.
In this paper we study the following nonlinear elliptic differential inequality with $p(x)$-Laplacian

\begin{equation}\label{eq_01}
  \left\{
  \begin{array}{lr}
    -\Delta_{p(x)}u-\lambda |u(x)|^{p(x)-2} u(x) \in \partial j(x, u(x))& \textrm{a.e. on  } \Omega,\\
    u=0 & \textrm{on}\ \partial \Omega,
  \end{array}
  \right.
\end{equation}
  where $p:\overline{\Omega} \rightarrow \mathbb{R}$ is a continuous function satisfying
\begin{equation} \label{pp}
  1<p^- \leqslant p(x) \leqslant p^+<N<\infty \textrm{ and } p^+\leqslant \hat{p}^*:=\frac{Np^-}{N-p^-}
\end{equation}
  where $p^-:=\inf\limits_{x \in \Omega} p(x)$ and  $p^+:=\sup\limits_{x \in \Omega} p(x)$. The operator 
\[
\Delta_{p(x)}u:= \textrm{div} \big( |\nabla u(x)|^{p(x)-2} \nabla u(x) \big)
\]
  is the so-called $p(x)$-Laplacian, which for $p(x)\equiv p$ becomes $p$-Laplacian.
  Problems with $\Delta_{p(x)}u$ are more complicated than with $\Delta_{p}u$
  because they are usually inhomogeneous and possesses "more nonlinearity". The function
$j(x,t)$ is locally Lipschitz
  in the $t$-variable
  and measurable in $x$-variable and by $\partial j(x,t)$ we denote the subdifferential with respect
  to the $t$-variable in the sense of Clarke \cite{Clarke}.

In problem (\ref{eq_01}) appears $\lambda$, for which we will assume that
  \begin{equation} \label{lamb2}
\lambda < \frac{p^-}{p^+} \lambda_* \quad \textrm{ and } \quad \lambda<\frac{(p^- -1)p^+}{(p^+-1)p^-}\lambda_*,
  \end{equation} 
 where $\lambda_*$
 is defined by
\begin{equation}\label{lambd}
  \lambda_*=\inf_{u \in W_0^{1,p(x)}(\Omega)\backslash \{0\}}
  \frac{\int_{\Omega}|\nabla u(x)|^{p(x)}dx}{\int_{\Omega}|u(x)|^{p(x)}dx}.
\end{equation}
 When we consider 
\[
\widetilde{p}:=\min \Big\{\frac{(p^- -1)p^+}{(p^+-1)p^-}, \frac{p^-}{p^+} \Big\},
\]
then 
\begin{equation} \label{dodatkowy}
\lambda<\widetilde{p} \lambda_*.
\end{equation}

 It may happen that $\lambda_*=0$ (see Fan-Zhang \cite{fanzhang}).\\


  Recently, the study of $p(x)$-Laplacian problems has attracted more and more attention. The hemivariational inequalities with Dirichlet or Neumann boundary condition have been considered by many authors. In particular in Ge-Xue \cite{ge} and  Qian-Shen \cite{qian1}, the following differential inclusion involving $p(x)$-Laplacian is studied

\[\label{eq_02}
  \left\{
  \begin{array}{lr}
    -\Delta_{p(x)}u \in \partial j(x, u(x))& \textrm{a.e. on  } \Omega,\\
    u=0 & \textrm{on}\ \partial \Omega,
  \end{array}
  \right.
\]
where $p$ is a continuous function satisfying (\ref{pp}). In the last paper the existence of two solutions of constant sign is proved. 

Also the study of variational problems is an interesing topic in recent years. For example in Fan-Zhang \cite{zhang} some sufficient conditions for the existence of solutions for the Dirichlet problem with $p(x)$-Laplacian is presented. Also in Ji \cite{ji}, the existence of three solutions for a differential equation is proved.

Finally we have papers with differential inclusions involving $p(x)$-Laplacian of the following type
\[\label{eq_03}
  \left\{
  \begin{array}{lr}
    -\Delta_{p(x)}u+\lambda |u(x)|^{p(x)-2} u(x) \in \partial j(x, u(x))& \textrm{a.e. on  } \Omega,\\
    u=0 & \textrm{on}\ \partial \Omega,
  \end{array}
  \right.
\]
where $\lambda >0$. 
In Ge-Xue-Zhou \cite{ge2}, authors proved sufficient conditions to obtain radial solutions for
  differential inclusions with $p(x)$-Laplacian. Differential inclusion with Neumann boundary condition were studied in
  Qian-Shen-Yang \cite{qian}
  and Dai \cite{dai}. The authors considered an inclusion involving a weighted function
  which is indefinite. In Dai \cite{dai}, the existence of infinitely many nonnegative solutions
  is proved.

  All the above mentioned papers deal with the so called hemivariational inequalities,
  i.e. the multivalued part is provided by the Clarke subdifferential of the nonsmooth potential
  (see e.g. Naniewicz-Panagiotopoulos \cite{nana}).

  In this paper we have the situation that $\lambda$ can be positive or negative (see (\ref{lamb2})). It is an extension of the theory considered in the above mentioned papers. 
  Our method is more direct and is based on the critical point theory
  for nonsmooth Lipschitz functionals due to Chang \cite{changg}.
  For the convenience of the reader in the next section we briefly present
  the basic notions and facts from the theory, which will be used in
  the study of problem (\ref{eq_01}).
  Moreover, we present the main properties of the general Lebesgue and variable Sobolev spaces.

\section{Mathematical preliminaries} \label{Prelim}

  Let $X$ be a Banach space and $X^*$ its topological dual.
  By $\|\cdot\|$ we will denote the norm in $X$ and by
  $\langle\cdot,\cdot\rangle$ the duality brackets for the pair $(X,X^*)$.
 In analogy with the directional derivative of a convex function,
  we introduce the notion of the generalized directional derivative of a locally Lipschitz function $f$ at $x \in X$ in the direction $h \in X$
  by
\[
    f^0 (x;h)=\limsup_{y \rightarrow x, \lambda \rightarrow 0} \frac{f(y+\lambda h)-f(y)}{\lambda}.
\]
  The function $h \longmapsto f^0 (x,h) \in \mathbb{R}$ is sublinear and 
  continuous so it is the support function of a nonempty, $w^*$-compact and convex set
\[
    \partial f(x)=\{x^* \in X^*: \langle x^*,h \rangle \leqslant f^0 (x,h) \textrm{ for all } h \in X\}.
\]
  The set $\partial f(x)$ is known as generalized or Clarke subdifferential of $f$ at $x$. If $f$ is convex, then $\partial f (x)$ coincides with the subdifferential in the sense of convex analysis. 

 Let $f:X\rightarrow\mathbb{R}$ be a locally Lipschitz function.
  From convex analysis it is well know that a proper, convex and lower
  semicontinuous function
  $g: X \rightarrow \overline{\mathbb{R}}=\mathbb{R}\cup \{ + \infty\}$
  is locally Lipschitz in the interior of its effective domain $\textrm{dom } g=\{x \in X:g(x)< \infty\}$. A point $x \in X$ is said to be a critical point of the locally
  Lipschitz function $f: X \rightarrow \mathbb{R}$, if $0 \in \partial f(x)$.

\begin{lemma}(Qian-Shen-Zhu \cite{qian3})
  Let $f, g: X\rightarrow\mathbb{R}$ be two locally Lipschitz functions. Then

\noindent
  (a) $f^0(x;h)=\max \{ \langle \xi,h \rangle: \xi \in \partial f(x)\}$;\\
\noindent
  (b) $(f+g)^0(x;h) \leq f^0(x;h)+g^0 (x;h)$;\\
\noindent
  (c) $(-f)^0(x;h)=f^0(x;-h) \textrm{ and } f^0(x,kh)=kf^0(x;h) \textrm{ for every } k>0$;\\
\noindent
  (d) the function $(x;h) \rightarrow f^0(x;h)$ is upper semicontinuous;
\end{lemma}

  From more details on the generalized subdifferential we refer to Clarke \cite{Clarke} and Gasi\'nski-Papageorgiou \cite{lg3}.

  We say that $f$ satisfies the "nonsmooth Palais-Smale condition" (nonsmooth PS-condition for short),
  if any sequence $\{x_n\}_{n \geqslant 1} \subseteq X$ such that $\{f(x_n)\}_{n \geqslant 1}$
  is bounded and $m(x_n)=\min\{ \|x^*\|_* :x^* \in \partial f(x_n)\}\rightarrow 0$
  as $n \rightarrow \infty$, has a strongly convergent subsequence.

  The first theorem is due to Chang \cite{changg} and extends to
  a nonsmooth setting the well known "mountain pass theorem" due
  to Ambrosetti -Rabinowitz \cite{ambro}.

\begin{theorem} \label{twierdzenie}
  If $X$ is a reflexive Banach space, $R:X \rightarrow \mathbb{R}$ is
  a locally Lipschitz functional satisfying PS-condition and for some
  $\rho>0$ and $y \in X$ such that $\|y\|>\rho$, we have
\[
    \max\{R(0),R(y)\}<\inf\limits_{\|x\|=\rho} \{R(x)\}=: \eta,
\]
   then R has a nontrivial critical point $x \in X$ such that
   the critical value $c=R(x) \geqslant \eta$ is characterized by the following minimax principle
\[
    c=\inf\limits_{\gamma \in \Gamma}\max\limits_{0 \leqslant \tau \leqslant 1}\{R(\gamma(\tau))\},
\]
  where $\Gamma=\{\gamma \in \mathcal{C}([0,1],X):\gamma(0)=0,\gamma(1)=y\}$.
\end{theorem}

  In order to discuss problem (\ref{eq_01}),
  we need to state some properties of the spaces $L^{p(x)}(\Omega)$ and $W^{1,p(x)}(\Omega)$,
  which we call generalized Lebesgue-Sobolev spaces (see Fan-Zhao \cite{fan,orlicz}).

  Let
\[
  E(\Omega)=\{u:\Omega\longrightarrow\mathbb{R}:\ u \textrm{ is measurable}\}.
\]
  Two functions in $E(\Omega)$ are considered to be one element
  of $E(\Omega)$, when they are equal almost everywhere. Firstly, we define the variable exponent Lebesgue space by
\[
    L^{p(x)}(\Omega)=\{u \in E(\Omega): \int_{\Omega} |u(x)|^{p(x)} dx<\infty \},
\]
  with the norm
\[
    \|u\|_{p(x)}=\|u\|_{L^{p(x)}(\Omega)}=\inf \Big\{\lambda>0: \int_{\Omega} \Big|\frac{u(x)}{\lambda}\Big|^{p(x)}dx \leqslant 1 \Big\}.
\]
  Next, the generalized Lebesgue-Sobolev space $W^{1,p(x)}(\Omega)$ is defined as
\[
    W^{1,p(x)}(\Omega) = \{u \in L^{p(x)}(\Omega):|\nabla u| \in L^{p(x)}(\Omega; \mathbb{R}^N)\}
\]
  with the norm
\[
    \|u\|=\|u\|_{W^{1,p(x)}(\Omega)} = \|u\|_{p(x)}+\|\nabla u\|_{p(x)}.
\]

Then $(L^{p(x)}(\Omega),\|\cdot\|_{p(x)})$ and $(W^{1,p(x)}(\Omega),\|\cdot\|)$ are separable and refelxive Banach spaces. By $W_0^{1,p(x)}(\Omega)$ we denote the closure of $C_0^\infty (\Omega)$ in $W^{1,p(x)}(\Omega)$.

\begin{lemma} [Fan-Zhao \cite{fan}]  \label{fan}
 If $\Omega \subset \mathbb{R}^N$ is an open domain, then

\noindent 
  (a) if $1 \leqslant q(x) \in \mathcal{C}(\overline{\Omega})$ and $q(x) \leqslant p^*(x)$ (respectively $q(x) < p^*(x)$) for any $x \in \overline{\Omega}$, where
\[
  p^*(x)=\left\{
  \begin{array}{ll}
   \frac{Np(x)}{N-p(x)} & p(x)<N\\
    \infty & p(x) \geqslant N,
  \end{array}
  \right.
\]
 then $W^{1,p(x)}(\Omega)$ is embedded continuously (respectively compactly) in $L^{q(x)}(\Omega)$;

\noindent
 (b) Poincar\'e inequality in $W_0^{1,p(x)}(\Omega)$ holds i.e., there exists a positive constant $c$ such that
 \[
  \|u\|_{p(x)} \leqslant c \|\nabla u\|_{p(x)}  \qquad \textrm{for all } u \in W_0^{1,p(x)}(\Omega);
\]

\noindent
 (c) $(L^{p(x)}(\Omega))^*=L^{p'(x)}(\Omega)$, where $\frac{1}{p(x)}+\frac{1}{p'(x)}=1$
   and for all $u \in L^{p(x)}(\Omega)$ and $v \in L^{p'(x)}(\Omega)$, we have
\[
  \int_{\Omega}|uv|dx\leqslant \Big(\frac{1}{p^-}+\frac{1}{
  {p'}^{-}}\Big)\|u\|_{p(x)}\|v\|_{p'(x)}.
\]
\end{lemma}

\begin{lemma}[Fan-Zhao \cite{fan}]\label{lemma2}
  Let $\varphi (u)=\int_{\Omega} |u(x)|^{p(x)} dx$ for $u \in L^{p(x)}(\Omega)$ and let
  $\{u_n\}_{n \geqslant 1} \subseteq L^{p(x)}(\Omega)$.

\noindent (a) for $u\neq 0$, we have
\begin{center} $\|u\|_{p(x)} =a \Longleftrightarrow \varphi (\frac{u}{a})=1$;\end{center}

\noindent (b) we have
\begin{center}$\|u\|_{p(x)}<1 \; \Longleftrightarrow \; \varphi(u)<1$;\end{center}

\begin{center}$\|u\|_{p(x)}=1 \; \Longleftrightarrow \;  \varphi(u)=1$;\end{center}

\begin{center}$\|u\|_{p(x)}>1 \; \Longleftrightarrow \; \varphi(u)>1$;\end{center}

\noindent (c) if $\|u\|_{p(x)}>1$, then
\begin{center}$\|u\|^{p^-}_{p(x)} \leqslant \varphi (u) \leqslant \|u\|^{p^+}_{p(x)}$;\end{center}

\noindent (d) if $\|u\|_{p(x)}<1$, then
\begin{center}$\|u\|^{p^+}_{p(x)} \leqslant \varphi (u) \leqslant \|u\|^{p^-}_{p(x)}$;\end{center}

\noindent (e) we have
\begin{center}$\lim\limits_{n \rightarrow \infty} \|u_n\|_{p(x)} =0 \; \Longleftrightarrow \; \lim\limits_{n \rightarrow \infty} \varphi(u_n) =0$;\end{center}

\noindent (f) we have
\begin{center}$\lim\limits_{n \rightarrow \infty}\|u_n\|_{p(x)} = \infty \; \Longleftrightarrow \; \lim\limits_{n \rightarrow \infty} \varphi (u_n) =\infty$.\end{center}

\end{lemma}

  Similarly to Lemma \ref{lemma2}, we have the following result.

\begin{lemma} [Fan-Zhao \cite{fan}] \label{lemma6}
  Let $\Phi (u) = \int_{\Omega} (|\nabla u(x)|^{p(x)}+|u(x)|^{p(x)})dx$
  for $u \in W^{1,p(x)}(\Omega)$ and let $\{u_n\}_{n \geqslant 1} \subseteq W^{1,p(x)}(\Omega)$.
  Then

\noindent (a) for $u \neq 0$, we have
\begin{center}$\|u\|=a \; \Longleftrightarrow \; \Phi (\frac{u}{a})=1$;\end{center}

\noindent (b) we have
\begin{center}$\|u\|<1 \; \Longleftrightarrow \; \Phi (u) <1$;\end{center}

\begin{center}$\|u\|=1 \; \Longleftrightarrow \; \Phi (u) =1$;\end{center}

\begin{center}$\|u\|>1 \; \Longleftrightarrow \; \Phi (u)>1$;\end{center}

\noindent (c) if $\|u\|>1$, then
\begin{center} $\|u\|^{p^-} \leqslant \Phi (u) \leqslant  \|u\|^{p^+}$;\end{center}

\noindent (d) if $\|u\|<1$, then
\begin{center}$\|u\|^{p^+} \leqslant \Phi (u) \leqslant  \|u\|^{p^-}$;\end{center}

\noindent (e) we have
\begin{center}$\lim\limits_{n \rightarrow \infty} \|u_n\|=0 \Longleftrightarrow \lim\limits_{n \rightarrow \infty} \Phi (u_n) =0$;\end{center}

\noindent (f) we have
\begin{center}$\lim\limits_{n \rightarrow \infty} \|u_n\| = \infty \Longleftrightarrow \lim\limits_{n \rightarrow \infty} \Phi (u_n) = \infty$.\end{center}

\end{lemma}

  Consider the following function
\[
    J(u)=\int_{\Omega} \frac{1}{p(x)}|\nabla u|^{p(x)} dx, \qquad \textrm{for all } u \in W_0^{1,p(x)}(\Omega).
\]
  We know that $J \in \mathcal{C}^1 (W_0^{1,p(x)}(\Omega))$ and $-\textrm{div}(|\nabla u|^{p(x)-2} \nabla u)$
  is the derivative operator of $J$ in the weak sense (see Chang \cite{chang}).
  We denote
\[
  A=J': W_0^{1,p(x)}(\Omega) \rightarrow (W_0^{1,p(x)}(\Omega))^*,
\]
  then
\begin{equation} \label{operator}
  \langle Au,v\rangle = \int_{\Omega} |\nabla u(x)|^{p(x)-2} (\nabla u(x), \nabla v(x)) dx
  \end{equation}
for all  $u, v \in W_0^{1,p(x)}(\Omega).$

\begin{lemma}[Fan-Zhang \cite{zhang}] \label{lemma3}
  If A is the operator defined above, then $A$ is a continuous,
  bounded and strictly monotone operator of type $(S_+)$ i.e.,
  if $u_n \rightarrow u$ weak in $W_0^{1,p(x)}(\Omega)$ and
  $\limsup\limits_{n \rightarrow \infty} \langle Au_n, u_n -u \rangle \leqslant 0$,
  implies that $u_n \rightarrow u$ in  $W_0^{1,p(x)}(\Omega)$.
\end{lemma}

\section{Existence of Solutions}
  We start by introducing our assumptions for the nonsmooth potential $j(x,t)$.\\

\noindent $H(j) \; j:\Omega  \times \mathbb{R} \rightarrow \mathbb{R}$ is a function such that $j(x,t)$ satisfies $j(x,0)=0$ almost everywhere on $\Omega$ and

\noindent \textbf{(i)} for all $t \in \mathbb{R}$, the function $\Omega \ni x \rightarrow j(x,t) \in \mathbb{R}$ is measurable;

\noindent \textbf{(ii)} for almost all $x \in \Omega$, the function $\mathbb{R} \ni t \rightarrow j(x,t) \in \mathbb{R}$ is locally Lipschitz;

\noindent \textbf{(iii)} for almost all $x \in \Omega$ and all $v \in \partial j(x,t)$, we have $|v| \leq a(x)+c_1|t|^{r(x)-1}$ with $a \in L_+^\infty (\Omega), c_1>0$ and $r \in \mathcal{C}(\overline{\Omega})$, such that $p^+ \leq r^+:=\max\limits_{x \in \Omega} r(x)<\widehat{p}^*:=\frac{Np^-}{N-p^-}$;

\noindent \textbf{(iv)} there exists $\mu >0$, such that
\[
  \limsup\limits_{|t| \rightarrow 0} \frac{j(x,t)}{|t|^{p(x)}} \leqslant -\mu,
\]
  uniformly for almost all $x \in \Omega$;

\bigskip

  As for the behaviour of $j$ in $+\infty$ and $-\infty$, we will consider one of the following two different conditions.\\

\noindent $H(j)_1$ \textbf{(v)} there exists $c>0$ that
\begin{equation} \label{war}
  \limsup\limits_{|t|\rightarrow \infty} \frac{v^*(x)t-j(x,t)}{|t|^{p(x)}} \leqslant -c,
\end{equation}
  uniformly for almost all $x \in \Omega$ and all $v^* (x) \in \partial j(x, t)$;

\noindent \textbf{(vi)} there exists $\overline{u} \in W^{1,p(x)}_0 (\Omega) \setminus \{0\}$, such that
\[
  \frac{1}{p^-}\int_{\Omega}|\nabla \overline{u} (x)|^{p(x)}dx
  +\frac{\lambda_-}{p^-}\int_{\Omega}|\overline{u}(x)|^{p(x)}dx \leqslant \int_{\Omega} j(x, \overline{u}(x))dx,
\]
  where $\lambda_-:=\max\{0, -\lambda\}$.

\bigskip

\noindent $H(j)_2$  \textbf{(v)} there exist constants $\nu>p^+$ and $M>0$ such that
\[
\nu j(x,t) \leq -j^0 (x,t;-t) \quad \textrm{and} \quad \textrm{ess\;inf} \;j(\cdot,t)>0,
\]
  for almost all $x \in \Omega$ and all $t \in \mathbb{R}$, such that $|t|>M$.

\begin{remark}
  Hypothesis $H(j)_1(vi)$ can be replaced by 

\noindent \textrm{(\bf vi)'}
  there exists $\overline{u} \in W^{1,p(x)}_0 (\Omega) \setminus \{0\}$, such that
\[
  \overline{c} \|\overline{u}\|^{p^+} \leqslant \int_{\Omega} j(x, \overline{u}(x))dx,
  \qquad \textrm{ if } \; \|\overline{u}\| \geqslant 1,
\]
  or
\[
  \overline{c} \|\overline{u}\|^{p^-} \leqslant \int_{\Omega} j(x, \overline{u}(x))dx,
  \qquad \textrm{ if } \; \|\overline{u}\| < 1,
\]
  where $\overline{c}:=\max\{\frac{1}{p^-},\frac{\lambda_-}{p^-}\}$.

Condition \textrm{(\bf vi)'} is more restrective, but easier to verify then condition \textrm{(\bf vi)}.
\end{remark}

\begin{remark} \label{uwaga}
  (i) If $H(j)$ and $H(j)_1$ hold, then for almost all $x \in \Omega$ and all $v^*(x) \in \partial j(x,t)$, we have  that $v^*(x)t-j(x,t) < 0$ for $|t|>M$. On the other hand, from the definition of subdifferential in the sense of Clarke, we have $-v^*(x)t \leq j^0(x,t;-t)$. Hence, there exists $M>0$ such that 
\begin{equation} \label{hip1}
j(x,t) \geqslant -j^0(x,t;-t),
\end{equation}
  for all $t$, such that $|t| > M$.

  (ii) If $H(j)$ and $H(j)_2$ hold, then for all $x \in \Omega$ and all $v^*(x) \in \partial j(x,t)$, we have 
\[
j(x,t) \leq \nu j(x,t) \leq -j^0 (x,t;-t) \leq v^*(x) t,
\]
  for all $t$, such that $|t| \geq M$ and $\nu>p^+$. So in fact, we know that
\begin{equation}\label{hip2}
\limsup\limits_{|t|\rightarrow \infty} \frac{v^*(x)t-j(x,t)}{|t|^{p(x)}} \geqslant 0,
\end{equation}
  uniformly for almost all $x \in \Omega$ and all $v^* (x) \in \partial j(x, t).$\\

  (iii) Hypotheses $H(j)_1$ and $H(j)_2$ (see (\ref{war}) and (\ref{hip2})) exclude each other.
\end{remark}

\begin{remark}
  The existence of nontrival solution for problem (\ref{eq_01}) was also considered in
  paper Barna\'s \cite{barnas}. In contrast to the last paper, instead of linear
  growth in $H(j)(iii)$ we consider the so-called sub-critical growth condition.   
  Moreover, condition $H(j)(iv)$ is more general. Also in hypothesis $H(j)_1$ and $H(j)_2$, we assume Tang-type condition - the most general one about behaviour in infinity.
\end{remark} 

\begin{lemma} \label{lemmakop}
  If hypotheses $H(j)$ and $H(j)_2$ hold, then\\
a) the function 
\[
f: k \rightarrow \frac{1}{k^{\nu}} j(x, kt)
\]
  is locally Lipschitz on $\mathbb{R}_+\backslash\{0\}$;\\
b) for all $\nu >  p^+$, there exist constants $l, M>0$ for which, we have 
\[
j(x,t) \geq l|t|^{\nu},
\] 
  for almost all $x \in \Omega$ and all $t$ such that $|t|>M$.
\end{lemma}

\begin{proof}
  Suppose that $U$ is a bounded set in $\mathbb{R}_+\backslash\{0\}$. From $H(j)(ii)$, we know that there exists $L>0$ such that 
\begin{equation}\label{new}
  |j(x,u_1)-j(x,u_2)| \leqslant L|u_1-u_2|, \; \textrm{ for all } \; u_1,u_2 \in U.
\end{equation}

  Now, let us fix $t \in \mathbb{R}$ and $x \in \Omega$. For some $k_1,k_2 \in U$, we have
\begin{eqnarray*}
  |f(k_1)-f(k_2)|=\Big| \frac{1}{k_1^{\nu}} j(x, k_1 t)-\frac{1}{k_2^{\nu} }j(x, k_2 t)\Big|\\
  \leqslant \frac{1}{k_1^{\nu}} |j(x, k_1 t)-j(x,k_2 t)|+\frac{|k_2^{\nu}j(x,k_2 t)|-k_1^{\nu} j(x, k_2 t)|}{k_1^{\nu}k_2^{\nu}}\\
=\frac{1}{k_1^{\nu}} |j(x, k_1 t)-j(x,k_2 t)|+\frac{|j(x,k_2 t)|}{k_1^{\nu}k_2^{\nu}} |k_2^{\nu}-k_1^{\nu}|.
\end{eqnarray*}

  Using (\ref{new}), we obtain 
\begin{eqnarray*}
  |f(k_1)-f(k_2)|\leqslant \qquad \qquad \qquad \qquad \qquad \qquad\\
\frac{L}{k_1^{\nu}}|k_1-k_2||t|+\frac{|j(x,k_2t|)}{k_1^{\nu}k_2^{\nu}}|k_1-k_2| \big| k_1^{\nu -1}k_2+k_1^{\nu-2}k_2^2+\ldots+k_1k_2^{\nu-1}\big|.
\end{eqnarray*}

  Since U is bounded, so we have
\[
  |f(k_1)-f(k_2)|\leqslant s|k_1-k_2|,
\]
  for some $s=s(x,t)$. This implies that the function $f$ is locally Lipschitz on $\mathbb{R}_+\backslash\{0\}$.\\

  Moreover, when we consider the subdifferential in the sense of Clarke of the function $f$, we obtain
\[
  \partial f(k)=\partial \Big( \frac{1}{k^{\nu}} j(x, kt)\Big) \subseteq \frac{-\nu}{k^{\nu+1}} j(x, kt) + \frac{1}{k^{\nu}} \partial j(x, kt)t,
\]
  for all $k \in \mathbb{R}_+\backslash\{0\}$. By virtue of the Lebourg mean value theorem for locally Lipschitz functions, for $k>1$ we can choose $\xi \in (1, k)$, such that
\begin{eqnarray} \label{vv}
f(k)-f(1) \in \Big[ \frac{-\nu}{\xi^{\nu+1}} j(x, \xi t) + \frac{1}{\xi^{\nu}} \partial j(x, \xi t)t \Big](k-1)\nonumber \\
=\frac{k-1}{\xi^{\nu+1}} (-\nu j(x, \xi t)+\partial j(x, \xi t) \xi t),
\end{eqnarray}
  for all $x \in \Omega$ and all $t$ such that $|t| > M$. From definition of subdifferential in the sense of Clarke, we have
\[
  \langle \eta, -\xi t\rangle \leqslant j^0(x, \xi t; -\xi t), \; \textrm{ for all } \eta \in \partial j(x, \xi t).
\]

  Combinig this with (\ref{vv}) and using $H(j)_2 (v)$, we obtain
\[
  \frac{k-1}{\xi^{\nu+1}} (-\nu j(x, \xi t)+\partial j(x, \xi t) \xi t) \geqslant \frac{k-1}{\xi^{\nu+1}} (-\nu j(x, \xi t)-j^0 (x, \xi t; -\xi t)) \geqslant 0.
\]

  Thus from (\ref{vv}), it follows that $f(k)-f(1)=\frac{1}{k^{\nu}} j(x,kt)-j(x,t) \geqslant 0$, and so
\[
  k^{\nu} j(x,t) \leqslant j(x, kt),
\]
  for all $x \in \Omega$, all $n \geqslant 1$ and $t$ such that $|t| > M.$

  Therefore,
\[
  j(x,t)=j(x,\frac{t}{M} M) \geqslant \frac{t^{\nu}}{M^{\nu}} j(x,M), \; \textrm{ for } t > M,
\]
  and
\[
  j(x,t)=j(x,\frac{t}{M} (-M)) \geqslant \frac{t^{\nu}}{M^{\nu}} j(x,-M), \; \textrm{ for } t < -M.
\]

  It follows that 
\[
  j(x,t) \geqslant l|t|^{\nu},
\]
  where $l=\frac{1}{M^{\nu}} \{ \textrm{ ess\;inf }j(\cdot, M), \textrm{ ess\;inf } j(\cdot, -M) \}>0$, for almost all $x \in \Omega$ and all $t$ such that $|t|>M$.

\end{proof}

   We introduce locally Lipschitz functional $R: W_0^{1,p(x)}(\Omega) \rightarrow \mathbb{R}$ defined by
\[
  R(u)=\int_{\Omega} \frac{1}{p(x)}|\nabla u(x)|^{p(x)}dx -\int_{\Omega}  
  \frac{\lambda}{p(x)}|u(x)|^{p(x)}dx - \int_{\Omega} j(x,u(x))dx,
\]
  for all $u \in W_0^{1,p(x)}(\Omega)$.

\begin{lemma}\label{PS}
If hypotheses $H(j)$ and $H(j)_1$ hold, and $\lambda \in (-\infty, \frac{(p^- -1)p^+}{(p^+-1)p^-}\lambda_*)$, then $R$ satisfies the nonsmooth PS-condition.
\end{lemma}

\begin{proof}

  Let $\{u_n\}_{n \geq 1} \subseteq W_0^{1,p(x)}(\Omega)$ be a sequence such that $\{R(u_n)\}_{n \geq 1}$ is bounded and $m(u_n) \rightarrow 0$ as $n \rightarrow \infty.$
  We will show that $\{u_n\}_{n \geq 1} \subseteq W_0^{1,p(x)}(\Omega)$ is bounded.

  Because $|R(u_n)|\leq M$ for all $n \geq 1$, we have
\begin{equation} \label{1111}
-M \leqslant \int_{\Omega} \frac{1}{p(x)} |\nabla u_n (x)|^{p(x)} dx-\int_{\Omega} \frac{\lambda}{p(x)}|u_n (x)|^{p(x)}dx - \int_{\Omega} j(x,u_n (x))dx.
\end{equation}

 Since $\partial R(u_n) \subseteq (W_0^{1,p(x)}(\Omega))^*$ is weakly compact, nonempty and the norm functional is weakly lower semicontinuous in a Banach space, then we can find $u_n^* \in \partial R(u_n)$ such that $||u_n^*||_*=m(u_n)$, for $n \geq 1$.

  Consider the operator $A:W_0^{1,p(x)}(\Omega) \rightarrow (W_0^{1,p(x)}(\Omega))^*$  defined by (\ref{operator}). Then, for every $n \geq 1$, we have
\begin{equation}\label{11ss}
u_n^*=Au_n-\lambda |u_n|^{p(x)-2} u_n - v_n^*,
\end{equation}
  where $v_n^* \in \partial \psi (u_n)\subseteq L^{p'(x)} (\Omega)$, for $n \geq 1$, with $\frac{1}{p(x)}+\frac{1}{p'(x)}=1$ and $\psi: W_0^{1,p(x)}(\Omega) \rightarrow \mathbb{R}$ is defined by $\psi (u_n)=\int\limits_{\Omega} j(x,u_n(x)) dx$. We know that, if $v_n^* \in \partial \psi (u_n)$, then $v_n^*(x) \in \partial j(x, u_n(x))$ (see Clarke \cite{Clarke}). 

  From the choice of the sequence $\{u_n\}_{n \geq 1} \subseteq W_0^{1,p(x)}(\Omega)$, at least for a subsequence, we have
\begin{equation}\label{syl}
|\langle u_n^*,w \rangle| \leq \varepsilon_n \quad \textrm{for all } w \in W^{1,p(x)}_0(\Omega),
\end{equation}
  with $\varepsilon_n \searrow 0$.

  Putting $w=u_n$ in (\ref{syl}) and using (\ref{11ss}), we obtain
\begin{equation} \label{236k}
-\varepsilon_n \leqslant - \int_{\Omega} |\nabla u_n(x)|^{p(x)} dx +\lambda \int_{\Omega} |u_n(x)|^{p(x)} dx +\int_{\Omega} v_n^* (x)u_n(x) dx.
\end{equation}
  
\bigskip

Now, let us consider two cases.\\

\noindent \textit{Case $1$. }

Let $\lambda \leqslant 0$. We define $\lambda_-:=\max \{0,-\lambda\}$.

  From (\ref{1111}) and (\ref{236k}), we have
\begin{eqnarray} \label{568}
-M-\varepsilon_n &\leqslant&  \Big( \frac{1}{p^-}-1 \Big) \int_{\Omega}|\nabla u_n (x)|^{p(x)} dx+ \lambda_- \Big( \frac{1}{p^-}-1\Big)\int_{\Omega} |u_n (x)|^{p(x)}dx \nonumber\\
&&\qquad \qquad +\int_{\Omega} v_n^* (x)u_n(x) dx - \int_{\Omega} j(x,u_n (x))dx.
\end{eqnarray}

  So we obtain that
\begin{eqnarray} \label{335}
 \lambda_- \Big( 1-\frac{1}{p^-}\Big) \int_{\Omega} |u_n (x)|^{p(x)}dx 
\leqslant \quad \quad \nonumber \\
M+\varepsilon_n +\int_{\Omega} v_n^* (x)u_n(x) dx - \int_{\Omega} j(x,u_n (x))dx.
\end{eqnarray}

  By virtue of hypotheses $H(j)_1(v)$, we know that there exist constant $c>0$, such that  
\[
\limsup\limits_{|t|\rightarrow \infty} \frac{v^*(x)t-j(x,t)}{|t|^{p(x)}} \leqslant -c,
\]
  uniformly for almost all $x \in \Omega$. So in particularly, there exists $L>0$ such that for almost all $x \in \Omega$ and all $|t| \geq L$, we have
\[
\frac{v^*(x)t-j(x,t)}{|t|^{p(x)}} \leqslant -\frac{c}{2}.
\]
  It immediately follows that
\begin{equation} \label{waznee}
v^*(x)t-j(x,t) \leqslant -\frac{c}{2}|t|^{p(x)}.
\end{equation}

  On the other hand, from the Lebourg mean value theorem (see Clarke \cite{Clarke}), for almost al $x \in \Omega$ and all $t \in \mathbb{R}$, we  can find $v(x) \in \partial j(x, k u(x))$ with $0<k<1$, such that
\[
|j(x,t)-j(x,0)| \leq |v(x)||t|.
\]

  So from hypothesis $H(j)(iii)$, for almost all $x \in \Omega$, we have
\[ \label{210}
|j(x,t)| \leq a(x)|t|+c_1|t|^{r(x)} \leq a(x)|t|+c_1|t|^{r^+}+c_2,
\]
  for some $c_2>0$. Then for almost all $x \in \Omega$ and all $t$ such that $|t|<L$, through (\ref{210}), it follows that
\begin{equation} \label{123}
|j(x,t)| \leq c_3,
\end{equation}
  for some $c_3>0$. Therefore, from (\ref{waznee}) and (\ref{123}) it follows that for almost all $x \in \Omega$ and all $ t \in \mathbb{R}$, we have
\begin{equation} \label{2040}
v^*(x)t-j(x,t) \leqslant -\frac{c}{2}|t|^{p(x)}+\beta,
\end{equation}
  for some $\beta>0.$

  We use (\ref{2040}) in (\ref{335}) and obtain
\[
 \lambda_- \Big( 1-\frac{1}{p^-}\Big) \int_{\Omega} |u_n (x)|^{p(x)}dx 
\leqslant M+\varepsilon_n - \frac{c}{2}\int_{\Omega}|u_n(x)|^{p(x)} dx + \int_{\Omega} \beta dx,
\]
  for all $n \geq 1$, which leads to
\[
 \Big[\lambda_- \Big( 1-\frac{1}{p^-} \Big) + \frac{c}{2} \Big] \int_{\Omega} |u_n (x)|^{p(x)}dx 
\leqslant M_1,
\]
  for some $M_1>0$.

  We know that $ \lambda_- \Big( 1-\frac{1}{p^-} \Big) + \frac{c}{2}>0$, so
\begin{equation}\label{ogr}
\textrm{the sequence } \{u_n\}_{n \geq 1} \subseteq L^{p(x)} (\Omega) \textrm{ is bounded}
\end{equation}
(see Lemma \ref{lemma2} (c) and (d)).\\

  Now, consider again (\ref{568}). We obtain
\[
 \Big( 1-\frac{1}{p^-}\Big) \int_{\Omega} |\nabla u_n (x)|^{p(x)}dx 
\leqslant M+\varepsilon_n +\int_{\Omega} v_n^* (x)u_n(x) dx - \int_{\Omega} j(x,u_n (x))dx.
\]
  In a similar way, by using (\ref{2040}) we have
\[
 \Big( 1-\frac{1}{p^-}\Big) \int_{\Omega} |\nabla u_n (x)|^{p(x)}dx 
\leqslant  M+\varepsilon_n - \frac{c}{2} \int_{\Omega} |u_n(x)|^{p(x)} dx+\int_{\Omega} \beta dx,
\]
  for all $n \geq 1$.

  So we obtain
\[
 \Big( 1-\frac{1}{p^-}\Big) \int_{\Omega} |\nabla u_n (x)|^{p(x)}dx 
\leqslant  M+\varepsilon_n+\int_{\Omega} \beta dx:=M_2,
\]

  Because $\Big( 1-\frac{1}{p^-}\Big)>0$, we have that
\begin{equation}\label{ogr1}
\textrm{the sequence } \{\nabla u_n\}_{n \geq 1} \subseteq L^{p(x)} (\Omega;\mathbb{R}^N) \textrm{ is bounded}
\end{equation}
  (see Lemma \ref{lemma2} (c) and (d)).

  From (\ref{ogr}) and (\ref{ogr1}), we have that
\begin{equation}\label{ogrs1}
\textrm{the sequence } \{u_n\}_{n \geq 1} \subseteq W_0^{1,p(x)} (\Omega) \textrm{ is bounded}
\end{equation}
  (see Lemma \ref{lemma6} (c) and (d)).\\

\noindent \textit{Case $2$. }

  Now, let $\lambda > 0$. 

  Again from (\ref{1111}) and (\ref{236k}), we have
\begin{eqnarray} \label{890}
-M-\varepsilon_n &\leqslant&  \Big( \frac{1}{p^-}-1 \Big) \int_{\Omega}|\nabla u_n (x)|^{p(x)} dx+ \lambda \Big(1- \frac{1}{p^+}\Big)\int_{\Omega} |u_n (x)|^{p(x)}dx \nonumber\\
&&\qquad \qquad +\int_{\Omega} v_n^* (x)u_n(x) dx - \int_{\Omega} j(x,u_n (x))dx.
\end{eqnarray}

  From the definition of $\lambda_*$ (see (\ref{lambd})), we have
\begin{equation} \label{lambdaa}
\lambda_* \int_{\Omega} |u_n(x)|^{p(x)}dx \leq \int_{\Omega} |\nabla u_n(x)|^{p(x)} dx,
\end{equation}
  for all $n \geq 1$.

  Using this fact in (\ref{890}), we have
\begin{eqnarray} \label{459}
\Big[ \lambda_* \Big( 1-\frac{1}{p^-} \Big) + \lambda \Big( \frac{1}{p^+}-1\Big)\Big]\int_{\Omega} |u_n (x)|^{p(x)}dx \leqslant \nonumber\\
 M+\varepsilon_n +\int_{\Omega} v_n^* (x)u_n(x) dx - \int_{\Omega} j(x,u_n (x))dx.
\end{eqnarray}

  In a similar way like in Case $1$, by using (\ref{2040}) in (\ref{459}), we obtain
 \[ 
\Big[ \lambda_* \Big( 1-\frac{1}{p^-} \Big) + \lambda \Big( \frac{1}{p^+}-1\Big) +\frac{c}{2}\Big] \int_{\Omega} |u_n (x)|^{p(x)}dx \leqslant M_3,
\]
  for some $M_3>0$.

  We know that 
\[ \lambda_* \Big( 1-\frac{1}{p^-} \Big) + \lambda \Big( \frac{1}{p^+}-1\Big) +\frac{c}{2}>0\]
  (see (\ref{lamb2})), so
\begin{equation}\label{ogr2}
  \textrm{the sequence } \{u_n\}_{n \geq 1} \subseteq L^{p(x)} (\Omega) \textrm{ is bounded}
\end{equation}
(see Lemma \ref{lemma2} (c) and (d)).

  Now, again from (\ref{890}), we have
\begin{eqnarray}\label{338899} 
\Big( 1-\frac{1}{p^-} \Big) \int_{\Omega}|\nabla u_n (x)|^{p(x)} dx \leqslant  M+\varepsilon_n+ \lambda \Big(1- \frac{1}{p^+}\Big)\int_{\Omega} |u_n (x)|^{p(x)}dx \nonumber \nonumber\\
\int_{\Omega} v_n^* (x)u_n(x) dx - \int_{\Omega} j(x,u_n (x))dx.
\end{eqnarray}
  Using (\ref{2040}) and (\ref{ogr2}) in (\ref{338899}), we obtain
\[
\Big( 1-\frac{1}{p^-} \Big) \int_{\Omega}|\nabla u_n (x)|^{p(x)} dx \leqslant  M_4,
\]
  for some $M_4>0$.

  Because $ 1-\frac{1}{p^-}>0$, we have that
\begin{equation}\label{ogr3}
\textrm{the sequence } \{\nabla u_n\}_{n \geq 1} \subseteq L^{p(x)} (\Omega;\mathbb{R}^N) \textrm{ is bounded}
\end{equation}
 (see Lemma \ref{lemma2} (c) and (d)).

  From (\ref{ogr2}) and (\ref{ogr3}), we have that
\[
\textrm{the sequence } \{u_n\}_{n \geq 1} \subseteq W_0^{1,p(x)} (\Omega) \textrm{ is bounded}
\]
(see Lemma \ref{lemma6} (c) and (d)).\\

 From Cases $1$ and $2$, we have that 
\begin{equation}\label{ogsuma}
\textrm{the sequence } \{u_n\}_{n \geq 1} \subseteq W_0^{1,p(x)} (\Omega) \textrm{ is bounded}.
\end{equation}

  Hence, by passing to a subsequence if necessary, we may assume that
\begin{equation}\label{1}
  \left.
  \begin{array}{ll}
   u_n \rightarrow u & \textrm{weakly in } W_0^{1,p(x)} (\Omega),\\
   u_n \rightarrow u & \textrm{in } L^{r(x)}(\Omega),
  \end{array}
  \right.
\end{equation}
  for any $r \in \mathcal{C}(\overline{\Omega})$, with $r^+=\max\limits_{x \in \Omega} r(x)< {\hat{p}}^*:=\frac{Np^-}{N-p^-}.$ 

  Putting $w=u_n-u$ in (\ref{syl}) and using (\ref{11ss}), we obtain
\begin{equation} \label{47}
  \Big|\langle Au_n,u_n-u\rangle - \lambda \int_{\Omega}  |u_n(x)|^{p(x)-2}u_n(x)(u_n-u)(x)dx -\int_{\Omega} v_n^*(x) (u_n-u)(x)dx\Big|\leq \varepsilon_n,
\end{equation}
  with $\varepsilon_n \searrow 0$.

  Using Lemma \ref{fan}(c), we see that
\begin{eqnarray*}
  &           & \lambda \int_{\Omega} |u_n(x)|^{p(x)-2}u_n(x)(u_n-u)(x)dx\cr
  & \leqslant & \lambda \Big(\frac{1}{p^-}+\frac{1}{p'^-}\Big) \|\, |u_n|^{p(x)-1}\|_{p'(x)}\|u_n-u\|_{p(x)},
\end{eqnarray*}
   where $\frac{1}{p(x)}+\frac{1}{p'(x)}=1$.

  We know that the sequence $\{u_n\}_{n \geqslant 1} \subseteq L^{p(x)}(\Omega)$
  is bounded, so using (\ref{1}), we can conclude that
\[
  \lambda \int_{\Omega} |u_n(x)|^{p(x)-2}u_n(x)(u_n-u)(x)dx \rightarrow 0 \quad \textrm{as } n \rightarrow \infty
\]
  and
\[
  \int_{\Omega} v_n^*(x) (u_n-u)(x)dx \rightarrow 0 \quad \textrm{as } n \rightarrow \infty.
\]
 If we pass to the limit as $n \rightarrow \infty$ in (\ref{47}), we have
\begin{equation} \label{68}
  \limsup\limits_{n \rightarrow \infty} \langle Au_n, u_n-u \rangle \leq 0.
\end{equation}

  So from Lemma \ref{lemma3}, we have that $u_n \rightarrow u$ in $W^{1,p(x)}_0(\Omega)$ as $ n \rightarrow \infty$. Thus $R$ satisfies the  (PS)-condition.
\end{proof}

\begin{lemma}\label{PS2}
If hypotheses $H(j)$ and $H(j)_2$ hold, then $R$ satisfies the nonsmooth PS-condition.
\end{lemma}

\begin{proof}
 Let $\{u_n\}_{n \geq 1} \subseteq W_0^{1,p(x)}(\Omega)$ be a sequence such that $\{R(u_n)\}_{n \geq 1}$ is bounded and $m(u_n) \rightarrow 0$ as $n \rightarrow \infty.$
  We will show that $\{u_n\}_{n \geq 1} \subseteq W_0^{1,p(x)}(\Omega)$ is bounded.

  Because $|R(u_n)|\leq M$ for all $n \geq 1$, we have
\begin{equation}\label{ogrr}
\int_{\Omega} \frac{1}{p(x)} |\nabla u_n (x)|^{p(x)} dx-\int_{\Omega} \frac{\lambda}{p(x)}|u_n (x)|^{p(x)}dx - \int_{\Omega} j(x,u_n (x))dx \leqslant M,
\end{equation}
  so
\begin{equation} \label{1112}
\int_{\Omega} |\nabla u_n (x)|^{p(x)} dx-\frac{\lambda_+ p^+}{p^-}\int_{\Omega} |u_n (x)|^{p(x)}dx - p^+ \int_{\Omega} j(x,u_n (x))dx \leqslant p^+ M,
\end{equation}
  where $\lambda_+:=\max \{ \lambda,0 \}$.
  From the choice of the sequence $\{u_n\}_{n \geq 1} \subseteq W_0^{1,p(x)}(\Omega)$, at least for a subsequence, we have
\begin{equation}\label{489nn}
|\langle u_n^*,w \rangle| \leq \varepsilon_n \quad \textrm{for all } w \in W^{1,p(x)}_0(\Omega),
\end{equation}
  with $\varepsilon_n \searrow 0$ and $u_n^*$ is like in (\ref{11ss}) in Lemma \ref{PS}.

  Taking $w=u_n$ in (\ref{489nn}) and using (\ref{11ss}), we obtain
\begin{equation} \label{236l}
- \int_{\Omega} |\nabla u_n(x)|^{p(x)} dx +\lambda \int_{\Omega} |u_n(x)|^{p(x)} dx +\int_{\Omega} v_n^* (x)u_n(x) dx \leqslant \varepsilon_n.
\end{equation}

  From the definition of subdifferential in the sense of Clarke, we have
\[
  \langle v_n^*,-u_n\rangle \leqslant j^0(x,u_n;-u_n), \quad \textrm{ for some }v_n^* \in \partial j(x, u_n(x)).
\]
  It follows that
\[
\langle v_n^*,u_n\rangle \geqslant -j^0(x,u_n;-u_n),
\]
  for all $n \geqslant 1$. Using this fact in (\ref{236l}), we obtain
\begin{equation} \label{237}
  -\int_{\Omega} |\nabla u_n(x)|^{p(x)} dx -\lambda_- \int_{\Omega} |u_n(x)|^{p(x)} dx -\int_{\Omega} j^0(x,u_n(x);-u_n(x)) dx \leqslant \varepsilon_n,
\end{equation}
  where $\lambda_-:= \max \{ - \lambda,0 \}$.
 
  Adding (\ref{1112}) and (\ref{237}), we have
\begin{eqnarray*}
  -\Big(\lambda_-+\frac{\lambda_+ p^+}{p^-} \Big) \int_{\Omega} |u_n(x)|^{p(x)} dx-\int_{\Omega}p^+ j(x,u_n (x))+j^0(x,u_n(x);-u_n(x)) dx \\
 \leqslant \varepsilon_n +Mp^+.
\end{eqnarray*}

  Thus,
\begin{eqnarray}\label{wy}
  -\int_{\Omega}\nu j(x,u_n (x))+j^0(x,u_n(x);-u_n(x)) dx
+ (\nu-p^+) \int_{\Omega} j(x,u_n (x))dx \nonumber\\
\leqslant\varepsilon_n +Mp^++\Big(\lambda_-+\frac{\lambda_+ p^+}{p^-} \Big) \int_{\Omega} |u_n(x)|^{p(x)} dx,
\end{eqnarray}
  for all $x \in \Omega$ and all $n \geqslant 1$, where $\nu > p^+ >p(x)$.

  Now, let us consider
\begin{eqnarray}\label{440}
  - \int_{\Omega}\nu j(x,u_n (x))+j^0(x,u_n(x);-u_n(x)) dx= \nonumber\\
  - \int_{\{|u_n|> M\} }\nu j(x,u_n (x))+j^0(x,u_n(x);-u_n(x)) dx\\
  - \int_{\{|u_n|\leqslant M\}}\nu j(x,u_n (x))+j^0(x,u_n(x);-u_n(x)) dx \nonumber
\end{eqnarray}

  From $H(j)(iii)$ and the Lebourg mean value theorem, similarly as in Lemma \ref{PS}, we can show that for almost all $x \in \Omega$ and all $|u_n| \leqslant M$, there exist constant $K_1>0$ such that
\[
  \nu j(x,u_n (x))+j^0(x,u_n(x);-u_n(x)) \leqslant K_1.
\]
  So 
\begin{equation}\label{441}
  -\int_{\{|u_n| \leqslant M\}}\nu j(x,u_n (x))+j^0(x,u_n(x);-u_n(x)) dx \geqslant -K_1.
\end{equation}

  From $H(j)_2(v)$, we know that there exists constants $M, K_2>0$ such that 
\begin{equation}\label{442}
  -\int_{\{|u_n| > M\}}\nu j(x,u_n (x))+j^0(x,u_n(x);-u_n(x)) dx \geqslant -K_2,
\end{equation}
  for almost all $x \in \Omega$ and $|u_n|> M.$

  Using (\ref{441}) and (\ref{442}), we obtain 
\begin{equation}\label{kk1}
  -\int_{\Omega}\nu j(x,u_n (x))+j^0(x,u_n(x);-u_n(x)) dx > -K,
\end{equation}
  where $K:=K_1+K_2>0$.\\

  Now, let us consider
\begin{eqnarray} \label{rr}
  (\nu-p^+) \int_{\Omega} j(x,u_n (x))dx=(\nu-p^+) \Big(\int_{\{|u_n| > M\}} j(x,u_n (x))dx\nonumber \\
  +\int_{\{|u_n| \leqslant M\}} j(x,u_n (x))dx \Big).\
\end{eqnarray}

  Again by the use of the Lebourg mean value theorem, inequality (\ref{rr}) and Lemma \ref{lemmakop}, we have
\begin{eqnarray}\label{now3}
  (\nu-p^+) \int_{\Omega} j(x,u_n (x))dx \geqslant (\nu-p^+) \Big( l\int_{\Omega} |u_n(x)|^{\nu} dx - K_1 \Big)\nonumber\\
 \geqslant K_3(||u_n||_{\nu}^{\nu} -K_1),
\end{eqnarray}
  for some $l, K_1, K_3>0$ and all $n \geqslant 1$.\\

  Using (\ref{kk1}) and (\ref{now3}) in (\ref{wy}), we have 
\begin{eqnarray}
  -K+K_3(||u_n||_{\nu}^{\nu} -K_1) \leqslant \varepsilon_n +Mp^++\Big(\lambda_-+\frac{\lambda_+ p^+}{p^-} \Big) \int_{\Omega} |u_n(x)|^{p(x)} dx. \nonumber
\end{eqnarray}

  Since $\nu > p^+>p(x)$ for all $x \in \Omega$, so it follows that
\begin{equation} \label{oo1}
  \textrm{the sequence } \{u_n\}_{n \geq 1} \subseteq L^{\nu} (\Omega) \textrm{ is bounded}.
\end{equation}

  For any $n \geqslant 1$ such that $||u_n||_{p(x)} \leqslant 1$ we have 
\[
  ||u_n||_{p(x)}^{p^+}< \int_{\Omega} |u_n(x)|^{p(x)} dx < \int_{\Omega} |u_n(x)|^{p^-} dx < ||u_n||_{\nu}^{\nu} \leqslant K_4,
\]
for some $K_4>0$ (see Lemma \ref{lemma2}).

  On the other hand,for any $n \geqslant 1$ such that $||u_n||_{p(x)} > 1$, we have
\[
  ||u_n||_{p(x)}^{p^-}< \int_{\Omega} |u_n(x)|^{p(x)} dx < \int_{\Omega} |u_n(x)|^{p^+} dx < ||u_n||_{\nu}^{\nu} \leqslant K_4.
\]
  Thus
\begin{equation} \label{ss}
\textrm{the sequence } \{u_n\}_{n \geq 1} \subseteq L^{p(x)} (\Omega) \textrm{ is bounded}.
\end{equation}

  Now, consider again (\ref{ogrr}) and multiply it by $\nu > p^+$ to obtain
\begin{equation}\label{009}
  \int_{\Omega} \frac{\nu}{p^+} |\nabla u_n (x)|^{p(x)} dx-\int_{\Omega} \frac{\lambda_+ \nu}{p^-}|u_n (x)|^{p(x)}dx - \nu\int_{\Omega} j(x,u_n (x))dx \leqslant \nu M,
\end{equation}
  for some $M>0$ and all $n \geqslant 1$. Adding (\ref{237}) and (\ref{009}), we obtain
\begin{eqnarray} 
  \int_{\Omega}\Big( \frac{\nu}{p^+}-1 \Big) |\nabla u_n(x)|^{p(x)} dx -\Big(\lambda_- +\frac{\lambda_+ \nu}{p^-}\Big)\int_{\Omega} |u_n(x)|^{p(x)} dx \nonumber\\
-\int_{\Omega} \nu j(x,u_n(x)+j^0(x,u_n(x);-u_n(x)) dx \leqslant \varepsilon_n+\nu M.\nonumber
\end{eqnarray}

  From (\ref{ss}) we know that the sequence $\{u_n\}_{n \geq 1} \subseteq L^{p(x)} (\Omega)$ is bounded and using the inequality (\ref{kk1}), we obtain
\[
  \Big( \frac{\nu}{p^+}-1 \Big) \int_{\Omega} |\nabla u_n(x)|^{p(x)} dx< K_5,
\]
  for some $K_5>0$ and all $n \geqslant 1$.\\

  Because $\frac{\nu}{p^+}-1>0$, so
\begin{equation} \label{oo2}
  \textrm{the sequence } \{ \nabla u_n\}_{n \geq 1} \subseteq L^{p(x)} (\Omega,\mathbb{R}^N) \textrm{ is bounded}.
\end{equation}

  From (\ref{ss}) and (\ref{oo2}), we have that
\begin{equation}
\textrm{the sequence } \{u_n\}_{n \geq 1} \subseteq W^{1,p(x)} (\Omega) \textrm{ is bounded}.
\end{equation}

The rest of proof is similar as the proof of Lemma \ref{PS}.
\end{proof}

\begin{lemma}\label{lemma5}
 If hypothesis $H(j)$ hold and $\lambda<\frac{p^-}{p^+}\lambda_*$,
  then there exist $\beta_1, \beta_2 >0$ such that for all
  $u \in W_0^{1,p(x)}(\Omega)$ with $\|u\|<1$, we have
\[
  R(u) \geqslant \beta_1 \|u\|^{p^+} - \beta_2 \|u\|^\theta,
\]
  with $p^+<\theta \leqslant \widehat{p}^{*}:=\frac{Np^-}{N-p^-}$.
\end{lemma}

\begin{proof}
  From hypothesis $H(j)(iv)$, we can find $\delta >0$,
  such that for almost all $x \in \Omega$ and all $t$ such that $|t|\leqslant \delta$, we have
\[
  j(x,t) \leqslant \frac{-\mu}{2}|t|^{p(x)}.
\]
  On the other hand, from hypothesis $H(j)(iii)$, we know
  that for almost all $x \in \Omega$ and all $t$ such that $|t|>\delta$, we have
\[
 |j(x,t)| \leq a_1|t|+c_1|t|^{r(x)},
\]
  for some $a_1, c_1>0$.
  Thus for almost all $x \in \Omega$ and all $t \in \mathbb{R}$ we have
\begin{equation}\label{porow}
  j(x,t) \leqslant\frac{-\mu}{2}|t|^{p(x)}+\gamma |t|^\theta,
\end{equation}
  with some $\gamma>0$ and $p^+<\theta<\widehat{p}^*$.

Let us consider two cases.

\bigskip

\noindent \textit{Case 1. }
Let $\lambda \leqslant 0$. 

By using (\ref{porow}), we obtain that
\begin{eqnarray*}
   R(u)&=&\int_{\Omega} \frac{1}{p(x)} |\nabla u (x)|^{p(x)} dx-\int_{\Omega} \frac{\lambda}{p(x)}|u (x)|^{p(x)}dx - \int_{\Omega} j(x,u (x))dx\\
  &\geqslant& \int_{\Omega} \frac{1}{p^+} |\nabla u (x)|^{p(x)} dx + \int_{\Omega} \frac{\mu}{2}|u(x)|^{p(x)} dx - \gamma \int_{\Omega}|u(x)|^\theta dx.
\end{eqnarray*}

So, we have
\[
R(u) \geqslant \beta_1 \Big[\int_{\Omega}|\nabla u(x)|^{p(x)}dx +\int_{\Omega} |u(x)|^{p(x)} dx\Big]-\gamma\|u\|^\theta_\theta,
\]
  where $\beta_1:=\min\{ \frac{1}{p^+},\frac{\mu}{2} \}.$

\bigskip

\noindent \textit{Case 2. }
Let $\lambda >0$.

Using (\ref{porow}) and (\ref{lambdaa}), we obtain that

\begin{eqnarray*}
   R(u)&=&\int_{\Omega} \frac{1}{p(x)} |\nabla u (x)|^{p(x)} dx-\int_{\Omega} \frac{\lambda}{p(x)}|u (x)|^{p(x)}dx - \int_{\Omega} j(x,u (x))dx\\
  &\geqslant& \int_{\Omega} \frac{1}{p^+} |\nabla u (x)|^{p(x)} dx-\int_{\Omega} \frac{\lambda}{p^-}|u(x)|^{p(x)}dx\\
  &&\hspace{2.7cm}+ \int_{\Omega} \frac{\mu}{2}|u(x)|^{p(x)} dx - \gamma \int_{\Omega}|u(x)|^\theta dx\\
  &\geqslant& \frac{1}{p^+} \int_{\Omega} |\nabla u(x)|^{p(x)} dx + \frac{\mu}{2} \int_{\Omega} |u(x)|^{p(x)}dx\\
 &&\hspace{2.7cm} -  \frac{\lambda}{\lambda_*p^-}\int_{\Omega} |\nabla u(x)|^{p(x)} dx -\gamma \|u\|^\theta_\theta\\
 &=& \Big(\frac{1}{p^+}-\frac{\lambda}{\lambda_*p^-}\Big)\int_{\Omega} |\nabla u (x)|^{p(x)} dx + \frac{\mu}{2} \int_{\Omega} | u(x)|^{p(x)}dx-\gamma \|u\|^\theta_\theta.
\end{eqnarray*}
  From the hypotheses, we have
\[
 \frac{1}{p^+}-\frac{\lambda}{\lambda_*p^-} > 0,
\]
  so
\[
  R(u)\geqslant \beta_1 \Big[\int_{\Omega}|\nabla u(x)|^{p(x)}dx +\int_{\Omega} |u(x)|^{p(x)} dx\Big]-\gamma\|u\|^\theta_\theta,
\]
  where $\beta_1:=\min\{ \frac{1}{p^+}-\frac{\lambda}{\lambda_*p^-},\frac{\mu}{2} \}.$

  As $\theta \leqslant p^*(x)=\frac{Np(x)}{N-p(x)}$, then $W^{1,p(x)}_0(\Omega)$
  is embedded continuously in $L^\theta(\Omega)$
  (see Lemma \ref{fan}(c)).
  So there exists $\varrho>0$ such that
\begin{equation}\label{48}
  \|u\|_\theta \leqslant \varrho\|u\| \qquad \textrm{for all } u \in W ^{1,p(x)}_0(\Omega).
\end{equation}
  Using (\ref{48}) and Lemma \ref{lemma6}(d), for all $u \in W ^{1,p(x)}_0(\Omega)$ with $\|u\|<1$, we have
\[
  R(u) \geqslant \beta_1 \|u\|^{p^+}- \beta_2\|u\|^\theta,
\]
  where $\beta_2=\gamma {\varrho}^{\theta}$.
\end{proof}
 
\noindent Using Lemmata \ref{PS}, \ref{PS2} and \ref{lemma5}, we can prove the following existence theorem for problem (\ref{eq_01}).

\begin{theorem} \label{big2}
If hypotheses $H(j)$ and $H(j)_1$ hold, $\lambda<\widetilde{p}\lambda_*$ (see (\ref{dodatkowy})), then problem (\ref{eq_01}) has a nontrival solution.
\end{theorem}

\begin{proof}
  From Lemma \ref{lemma5} we know that there exist $\beta_1,\beta_2>0$, such that for all $u\in W_0^{1,p(x)} (\Omega)$ with $||u||<1$, we have
\[
  R(u) \geq \beta_1 ||u||^{p^+}- \beta_2||u||^\theta = \beta_1||u||^{p^+} \Big( 1-\frac{\beta_2}{\beta_1}||u||^{\theta - p^+}\Big).
\]
  Since $p^+ <\theta$, if we choose $\rho \in (0, 1)$ small enough, we will have that $R(u)\geq L>0$, for all $u \in W_0^{1,p(x)}(\Omega)$, with $||u||=\rho$ and some $L>0$.\\

  Now, let $\overline{u} \in W^{1,p(x)}_0 (\Omega)$ be as in hypothesis $H(j)_1 (vi)$. We have
\begin{eqnarray*}
  R(\overline{u})&    =    &\int_{\Omega} \frac{1}{p(x)} |\nabla \overline{u} (x)|^{p(x)} dx-\int_{\Omega}
     \frac{\lambda}{p(x)}|\overline{u} (x)|^{p(x)}dx - \int_{\Omega} j(x,\overline{u} (x))dx\\
                 &\leqslant&\frac{1}{p^-} \int_{\Omega} |\nabla \overline{u} (x)|^{p(x)} dx +\frac{\lambda_-}{p^{-}}
     \int_{\Omega} |\overline{u}(x)|^{p(x)} dx-\int_{\Omega} j(x,\overline{u} (x))dx.\\
\end{eqnarray*}

  From hyphothesis $H(j)_1(vi)$, we get
  $R(\overline{u}) \leqslant 0$.
  This permits the use of Theorem \ref{twierdzenie}
  which gives us $u \in W_0^{1,p(x)}(\Omega)$ such that
  $R(u)>0 = R(0)$ and $0 \in \partial R(u)$.
  From the last inclusion we obtain
\[
  0=Au-\lambda |u|^{p(x)-2}u-v^*,
\]
  where $v^* \in \partial \psi(u).$ Hence
\[
  Au=\lambda |u|^{p(x)-2}u+v^*,
\]
  so for all $v \in \mathcal{C}_0^\infty (\Omega)$, we have $\langle Au,v \rangle = \lambda \langle |u|^{p(x)-2}u,v\rangle +\langle v^*,v \rangle$. 

  So we have
\begin{eqnarray*}
  &   & \int_{\Omega} |\nabla u(x)|^{p(x)-2}(\nabla u(x), \nabla v(x))_{\mathbb{R}^N}dx\cr
  & = & \int_{\Omega}\lambda |u(x)|^{p(x)-2}u(x)v(x)dx+\int_{\Omega}v^*(x)v(x) dx,
\end{eqnarray*}
  for all $v \in \mathcal{C}_0^\infty(\Omega)$.

  From the definition of the distributional derivative we have
\begin{equation}
  \left\{
  \begin{array}{lr}
    -\textrm{div} \big( |\nabla u(x)|^{p(x)-2} \nabla u(x) \big)=\lambda |u(x)|^{p(x)-2} u(x) +v(x)& \textrm{in } \Omega,\\
    u=0 & \textrm{on}\ \partial \Omega,
  \end{array}
  \right.
\end{equation}
  so
\begin{equation}
  \left\{
  \begin{array}{lr}
    -\Delta_{p(x)}u-\lambda |u(x)|^{p(x)-2} u(x)\in \partial j(x, u(x))& \textrm{in } \Omega,\\
    u=0 & \textrm{on}\ \partial \Omega.
  \end{array}
  \right.
\end{equation}
  Therefore $u \in W_0^{1,p(x)}(\Omega)$ is a nontrivial solution of (\ref{eq_01}).
\end{proof}

\begin{theorem} \label{big1}
If hypotheses $H(j)$ and $H(j)_2$ hold, $\lambda <\frac{p^-}{p^+} \lambda_*$, then problem (\ref{eq_01}) has a nontrival solution.
\end{theorem}

\begin{proof}
  From Lemma \ref{lemma5} we know that there exist $\beta_1,\beta_2>0$, such that for all $u\in W_0^{1,p(x)} (\Omega)$ with $||u||<1$, we have
\[
  R(u) \geq \beta_1 ||u||^{p^+}- \beta_2||u||^\theta = \beta_1||u||^{p^+} \Big( 1-\frac{\beta_2}{\beta_1}||u||^{\theta - p^+}\Big).
\]
  Since $p^+ <\theta$, if we choose $\rho \in (0, 1)$ small enough, we will have that $R(u)\geq L_1>0$, for all $u \in W_0^{1,p(x)}(\Omega)$, with $||u||=\rho$ and some $L_1>0$.\\
  Using Lemma \ref{lemmakop}(b), for any  $u \in W^{1,p(x)}_0 (\Omega) \backslash \{0\}$ we have
\begin{eqnarray*}
  R(tu)&=&\int_{\Omega} \frac{1}{p(x)} |\nabla t u (x)|^{p(x)} dx-\int_{\Omega} \frac{\lambda}{p(x)}|t u (x)|^{p(x)}dx - \int_{\Omega} j(x,t u(x))dx\\
&\leqslant&t^{p^+} \Big(\frac{1}{p^-} \int_{\Omega} | \nabla u (x)|^{p(x)} dx +\frac{\lambda_-}{p^{-}} \int_{\Omega} |u(x)|^{p(x)} dx \Big)-\int_{\Omega} j(x,t u(x))dx\\
&\leqslant&\overline{c} \cdot t^{p^+} \big(\int_{\Omega} (|\nabla u (x)|^{p(x)}+|u(x)|^{p(x)})dx\big) - l \cdot t^{\nu} \int_{\Omega} |u(x)|^{\nu}dx+L_2,
\end{eqnarray*}
  where $\lambda_-:=\max\{0, - \lambda\}$, $\overline{c}=\max\{ \frac{1}{p^-},\frac{\lambda_-}{p^-}\}$ and $\nu >p^+.$\\

  Because $\nu>p^+$, we get that
$R(t u) \rightarrow -\infty$ when $t \rightarrow \infty$. This permits the use of Theorem \ref{twierdzenie} which gives us $u \in W_0^{1,p(x)}(\Omega)$ such that $R(u)>0 = R(0)$ and $0 \in \partial R(u)$. \\

  Therefore $u \in W_0^{1,p(x)}(\Omega)$ is a nontrival solution of (\ref{eq_01}).
\end{proof}

\begin{remark}
  A nonsmooth potential satisfying hypothesis $H(j)$ and $H(j)_1$ is for example the one given by the following function:
\[
  j_1(x,t)= 
   \left\{
     \begin{array}{lcc}
      -\mu|t|^{p(x)} & \textrm{if} & |t| \leqslant 1,\\
      (\mu+\sigma -|2|^{q^+})|t|-2 \mu-\sigma+|2|^{q^+} & \textrm{if} & 1<|t|\leqslant 2,\\
      \sigma-|t|^{q^+} & \textrm{if} & |t| > 2,\\
     \end{array}
    \right.
\]
  with $\mu, \sigma>0$ and continuous function $p:\overline{\Omega} \rightarrow \mathbb{R}$ which satisfies  $1<p^- \leqslant p(x) \leqslant p^+<q^- \leqslant q(x) \leqslant q^+<\infty $
 and $ p^+\leqslant \widehat{p}^*.$\\

  A nonsmooth potential satisfying hypothesis $H(j)$ and $H(j)_2$ is for example the one given by the following function:
\[
  j_2(x,t)= 
   \left\{
     \begin{array}{lcc}
      -\mu|t|^{p(x)} & \textrm{if} & |t| \leqslant 1,\\
      |t|^{q^+}-\mu-1 & \textrm{if} & |t| > 1,\\
     \end{array}
    \right.
\]
  with $\mu>0$ and continuous functions $p, q:\overline{\Omega} \rightarrow \mathbb{R}$ which satisfies $1<p^- \leqslant p(x) \leqslant p^+<q^- \leqslant q(x) \leqslant q^+<\infty $
 and $ p^+\leqslant \widehat{p}^*.$\\

Of course $j_1$ does not satisfy $H(j)_2$ and $j_2$ does not satisfy $H(j)_1$ (see Remark \ref{uwaga}).
\end{remark}


\begin{thebibliography}{99}

\bibitem{ambro} A.Ambrosetti, P.H.Rabinowitz, \textit{Dual Variational Methods in Critical Point Theory and Applications}, J. Funct. Anal. 14 (1973), 349-381.

\bibitem{barnas} S.Barna\'s, {\it Existence result for hemivariational inequality involving $p(x)$-Laplacian}, Opuscula Mathematica, to appear.

\bibitem{chang} K.C.Chang, \textit{Critical Point Theory and Applications}, Shanghai Scientific and Technology Press, Shanghai, 1996.

\bibitem{changg} K.C.Chang, \textit{Variational Methods for Nondifferentiable Functionals and their Applications to Partial Differential Equations}, J. Math. Anal. Appl. 80 (1981), 102-129.

\bibitem{Clarke} F.H.Clarke, \textit{Optimization and Nonsmooth Analysis}, Wiley, New York,1993.

\bibitem{dai} G.Dai, \textit{Infinitely many solutions for a hemivariational inequality involving the $p(x)$-Laplacian}, Nonlinear Anal. 71 (2009), 186-195.

\bibitem{fansam} X.Fan, \textit{Eigenvalues of the $p(x)$-Laplacian Neumann problems}, Nonlinear Anal. 67 (2007), 2982-2992.


\bibitem{zhang} X.Fan, Q.Zhang, \textit{Existence of solutions for $p(x)$-Laplacian Dirichlet problem}, Nonlinear Anal. 52 (2003), 1843-1853.

\bibitem{fanzhang} X.Fan, Q.Zhang, D.Zhao, \textit{Eigenvalues of $p(x)$-Laplacian Dirichlet problem}, J. Math. Anal. Appl. 302 (2005), 306-317.

\bibitem{fan} X.Fan, D.Zhao, \textit{On the generalized Orlicz - Sobolev space $W^{k,p(x)}(\Omega)$}, J. Gansu Educ. College 12 (1) (1998), 1-6.

\bibitem{orlicz} X.Fan, D.Zhao, \textit{On the spaces $L^{p(x)}(\Omega)$ and $W^{m,p(x)}(\Omega)$}, J. Math. Anal. Appl. 263 (2001), 424-446.

\bibitem{lg} L.Gasi\'nski, N.S.Papageorgiou, \textit{Nonlinear hemivariational inequalities at resonance}, Bull. Austr. Math. Soc, 60:3 (1999), 353-364.

\bibitem{lg1} L.Gasi\'nski, N.S.Papageorgiou, \textit{Solutions and Multiple Solutions for Quasilinear Hemivariational Inequalities at Resonance}, Proc. Roy. Soc. Edinb., 131A:5 (2001), 1091-1111.

\bibitem{lg2} L.Gasi\'nski, N.S.Papageorgiou, \textit{An existance theorem for nonlinear hemivariational inequalities at resonance},  Bull. Austr. Math. Soc, 63:1 (2001), 1-14.

\bibitem{lg3} L.Gasi\'nski, N.S.Papageorgiou, \textit{Nonlinear Analysis: Volume 9}, Series in Mathematical Analysis and Applications (2006).

\bibitem{ge} B.Ge, X.Xue, \textit{Multiple solutions for inequality Dirichlet problems by the $p(x)$-Laplacian}, Nonlinear Anal. 11 (2010), 3198-3210.

\bibitem{ge2} B.Ge, X.Xue, Q.Zhou, \textit{The existence of radial solutions for differential inclusion problems in $\mathbb{R}^N$ involving the $p(x)$-Laplacian}, Nonlinear Anal. 73 (2010), 622-633.


\bibitem{ji} Ch. Ji, \textit{Remarks on the existence of three solutions for the $p(x)$-Laplacian equations}, Nonlinear Anal. 74 (2011), 2908-2915.

\bibitem{ko} N.Kourogenic, N.S.Papageorgiou, \textit{Nonsmooth critical point theory and nonlinear elliptic equations at resonance}, J. Aust. Math. Soc. 69 (2000), 245-271.

\bibitem{nana} Z.Naniewicz, P.D.Panagiotopoulos, \textit{Mathematical theory of hemivariational inequalities and applications}, Marcel-Dekker, New York (1995).

\bibitem{qian1} Ch.Qian, Z.Shen, \textit{Existence and multiplicity of solutions for $p(x)$-Laplacian equation with nonsmooth potential}, Nonlinear Anal. 11 (2010), 106-116.

\bibitem{qian} Ch.Qian, Z.Shen, M.Yang, \textit{Existence of solutions for $p(x)$-Laplacian nonhomogeneous Neumann problems with indefinite weight}, Nonlinear Anal. 11 (2010), 446-458.

\bibitem{qian3} Ch.Qian, Z.Shen, J.Zhu, \textit{Multiplicity results for a differential inclusion problem with non-standard growth}, J. Math. Anal. Appl. 386 (2012), 364-377.

\end{thebibliography}
\end{document}